\documentclass{amsart}
\usepackage{amsthm}
\usepackage{geometry}
\usepackage{amsmath,amssymb}
\usepackage{graphicx}
\newtheorem{theorem}{Theorem}[section]
\numberwithin{equation}{section}
\newtheorem{lemma}[theorem]{Lemma}
\newtheorem{proposition}[theorem]{Proposition}

\newtheorem{remark}[theorem]{Remark}

\newtheorem{notation}[theorem]{Notation}

\newcommand{\N}{{\mathbb N}}
\newcommand{\R}{{\mathbb R}}

\newcommand{\Z}{{\mathbb Z}}

\begin{document}

\title{Goldbach versus de Polignac numbers}

\author{Jacques Benatar}
\address{School of Mathematics, Tel Aviv University, Tel Aviv, Israel.}
\email{benatar@mail.tau.ac.il}

\begin{abstract} In this note we use recent developments in sieve theory to highlight the interplay between Goldbach and de Polignac numbers. Assuming that the primes have level of distribution greater than $1/2$, we show that at least one of two nice properties holds. Either consecutive Goldbach numbers lie within a finite distance from one another or else the set of de Polignac numbers has full density in $2 \N$. Using very similar techniques we give a conditional proof that the set of limit points of the sequence of normalised prime gaps $(p_{n+1}-p_n)/ \log p_n$ has  density at least $2/3$ in the positive reals. 
\end{abstract}

 \maketitle 
\section{Introduction}
Let $\mathcal{P}$ denote the set of prime numbers and write $p_n$ for its $n$-th member.  Given an admissible tuple of integers $\mathcal{H}=\left\{h_1,...,h_k\right\}$, the Hardy-Littlewood $k$-tuple prime conjecture is the assertion that 
$$\left\{n+h_1,...,n+h_k \right\} \subset \mathcal{P}$$
for infinitely many integers $n$. The problem has seen a number of breakthroughs over the past two decades and these efforts spawned an international collaboration known as the the Polymath8 project \cite{Poly}. Assuming the generalised Elliott-Halberstam conjecture, it was demonstrated that any admissible configuration $\left\{n+h_1,n+h_2,n+h_3 \right\}$ contains at least two primes for infinitely many values of $n$. These ideas can be applied to Goldbach numbers, that is to say, positive integers which are expressible as the sum of two primes. Fixing some large natural number $N$ one considers the collection $\left\{n,n+2,N-n \right\}$ and in this manner it can be shown, under suitable hypotheses, that at least one of the following statements must hold \cite[Theorem 1.5]{Poly}:
\begin{itemize}
\item[(i)] There are infinitely many twin primes.
\item[(ii)] One has $g_{n+1}-g_n \leq 4$ for all sufficiently large $n$.
\end{itemize} 
Here $g_n$ denotes the $n$-th Goldbach number. In this paper we prove a result of the same nature. To state the theorem, we recall that $m \in \N$ is said to be a de Polignac number if there exist infinitely many pairs of primes $(p,p')$ such that $p-p'=m$. Let $\mathcal{D}$ denote the set of de Polignac numbers.
\begin{theorem}\label{GvsP}
Assume that the primes have level of distribution $\theta >1/2$, which is to say EH$[\theta]$ holds for some $\theta \in(1/2,1)$ (see section \ref{sievesection} below). Then at least one of the following statements must hold:
\begin{itemize}
\item[(i)] There exists an absolute constant $C>0$ so that $g_{n+1}-g_n \leq C$ for all sufficiently large $n$.
\item[(ii)] The set $\mathcal{D}$ has full asymptotic density in the even numbers and more precisely
\begin{align}\label{Ddensity}
\left| \mathcal{D}^c \cap [0,N]\right| \leq N^{\kappa}
\end{align}
for all large $N$ and some $\kappa <1$ depending only on $\theta$. 
\end{itemize}
\end{theorem}
This result may be compared with Pintz' \cite[Theorem 1.4]{Pintz0} as well as \cite[Theorem 2.6]{Pintz} where, among other interesting facts, it is shown \textsl{unconditionally} that the set $\mathcal{D}$ has positive lower density in $2\N$. We also note that our statement holds for any level of distribution $\theta>1/2$ while the Polymath theorem uses the full strength of the Elliott-Halberstam Conjecture. The proof of Theorem \ref{GvsP} relies on an auto-correlation estimate for Maynard's sieve weight (see \cite{May}); such auto-correlation results have already been applied to the study of normalised prime gaps, that is to say, the sequence $s_n=(p_{n+1}-p_n)/ \log p_n$ where $p_n$ denotes the $n$-th prime. We will take this opportunity to discuss a conditional density estimate for the set $\mathcal{L}$, consisting of all limit points of $(s_n)_{n \in \N}$. Based on Cramer's probabilistic model of the primes, it is plausible that the sequence $s_n$ is distributed according to a Poisson process, i.e., for each $0< \alpha < \beta$ one expects that \cite{Gall}, \cite{Sound}, 
\begin{equation}\label{poisson}
\frac{1}{N} |\left\{ n \leq N: s_n \in (\alpha, \beta) \right\} | \sim \int_{\alpha}^{\beta} e^{-t} \ dt.
\end{equation}
Unable to prove this strong conjecture, several authors (e.g. \cite{BakFrei}, \cite{GL}, \cite{HM},\cite{Pintz2}) have turned their attention to various large scale and extremal questions regarding the set $\mathcal{L}$. For example, writing $m$ for Lebesgue measure, one might ask to bound the asymptotic lower density 
$$\underline{d}(\mathcal{L})=\liminf_{T \rightarrow \infty} \frac {m([0,T] \cap \mathcal{L}) }{T}$$ 
from below. Recent contributions to this question can be found in the breakthrough paper of Banks, Freiberg and Maynard \cite{Banks} and, subsequently, in \cite{Pintz2}, \cite{BakFrei} and \cite{Mer}. Assuming a special form of the Elliott-Halberstam conjecture (see section \ref{normalisedsection} below), we will give a conditional improvement of the density results obtained in the aforementioned works.
\begin{proposition}\label{Ldensity}
Assuming the conjecture $EH^{*}[\theta]$ holds for some $\theta \in(2/3,1)$, the limit set $\mathcal{L}$ obeys the (ineffective) estimate
\begin{equation}\label{lowerdensityL}
\underline{d}(\mathcal{L}) \geq 2/3.
\end{equation}
\end{proposition}
The proof of Proposition \ref{Ldensity} relies on a combinatorial result concerning sum-free subsets of $\R^{+}=[0,\infty)$; we were unable to find this fact in the literature\footnote{Proposition \ref{sumfree} was proven independently by J. Merikoski in unpublished work.}. Recall that the set $A \subset \R^{+}$ is called sum-free if the equation $x+y=z$ has no solutions in $A^3$. We also define the asymptotic upper density of the set $A$ to be  $\overline{d}(A)=\limsup_{T \rightarrow \infty} (m([0,T] \cap \mathcal{L}) /T$.
\begin{proposition}\label{sumfree}
For any Lebesgue measurable sum-free set $E \subset \R^{+}$, one has that $\overline{d}(E) \leq 1/3$. 
\end{proposition} 

\begin{remark} Proposition \ref{sumfree} is optimal in the sense that the set 
$$3\N+(1,2)=\bigcup_{k \geq 0} (3k+1,3k+2)$$ 
is sum-free with (upper) density $1/3$.
\end{remark}
\begin{notation}
We give some standard notation that will be used throughout the paper. For functions $f$ and $g$ we will use the symbols $f\ll g$ and $f=O(g)$ interchangeably to express Landau's big O symbol. A subscript of the form $\ll_{\eta}$ means the implied constant may depend on the quantity $\eta$. 
The statement $f \sim g$ means $f$ and $g$ are asymptotically equivalent, i.e., $\lim_{x \rightarrow \infty} f(x)/g(x)=1$. 
Given a natural number $m$, we write $P^{+}(m)$ for its largest prime divisor. A sum with superscript $\sum^{\flat}$ indicates a restriction to squarefree variables.
We reserve the letter $\mu$ for the M\"obius function and write $[N]=\left\{1,2,...,N\right\}$ for any natural number $N$. The letter $p$ always refers to a prime number and $\log_r n$ denotes the $r$-fold iterated logarithm of $n$. Finally, we write $(d,e)=\gcd(d,e)$ and let $[s_1,,...,s_r]=lcm(s_1,...,s_r)$ denote the least common multiple of the positive integers $s_1,...,s_r$.  
\end{notation}
{\bf Special Notation} To simplify notation we will write $r(N,k)=o_k(1)$ when $\lim_{k \rightarrow \infty} r(N,k)=0$, independently of $N$.

\section{Setting up the sieve}\label{sievesection}
\subsection{Preliminaries}
Let $N>0$ be a large, growing parameter and set 
\begin{equation}\label{WQdef}
W=\prod_{p \leq w} p, \qquad w:= \log_3 N.
\end{equation}
Throughout this paper, the natural number $k$ may be viewed as arbitrarily large but fixed or, alternatively, growing very slowly with $N$. It will be convenient to choose a $\psi: \N \rightarrow (0,\infty)$ tending to zero in such a way that $\psi(k) \log k \rightarrow \infty$. We write 
\[ \tilde{\psi}(k)=\psi(k) \log k. \] 
In order to obtain clusters of primes in bounded intervals we first select an admissible $k$-tuple of distinct integer shifts 
$\mathcal{H}=\left\{ h_1,...,h_k \right\}$, with each shift $h_j$ satisfying $\mathcal{P}(2k+1) | h_j$. 
Recall that $\mathcal{H}$ is said to be admissible if, for each prime $p$, the set $\mathcal{H}$ does not occupy each residue class modulo $p$. 
Let us assume that each $h \in \mathcal{H}$ satisfies the bound $|h|\leq w$ and extend $\mathcal{H}$ to an admiissible $2k$-tuple 
$$\overline{\mathcal{H}}=\left\{ h_1,...,h_{2k} \right\}$$ 
by choosing $N$ to be a multiple of $\mathcal{P}(2k+1)$ and introducing the \textsl{dual} shift $h_j=h_{j-k}-2N$ for each $j=k+1,...,2k$. 
Given any prime $p>w$, it is not hard to see that $p|(h_{\ell_1}-h_{\ell_2})$ for at most one pair $1 \leq \ell_1 < \ell_2 \leq 2k$. 
As a result we get the decomposition 
$$\left\{ p \in \mathcal{P}:p>w \right\}= P_{0} \cup \dot{\bigcup}_{1 \leq \ell_1 < \ell_2 \leq 2k} P_{\ell_1, \ell_2 },$$
where, for each pair $1 \leq \ell_1 < \ell_2 \leq 2k$,  
\begin{equation}\label{Pl1l2}
P_{\ell_1, \ell_2 }=P_{\ell_1, \ell_2 }(\overline{\mathcal{H} })=\left\{ p>w:p|(h_{\ell_1}-h_{\ell_2}) \right\}
\end{equation}
and $P_0$ consists of all remaining primes $p > w$ which do not divide any of the differences $h_{\ell_1}-h_{\ell_2}$. Setting
\begin{equation}\label{Qdef}
Q_{\ell_1,\ell_2}= \prod_{p \in P_{\ell_1, \ell_2 }}  p,
\qquad  \qquad  Q=\prod_{\substack{(\ell_1,\ell_2) \\ 1 \leq \ell_1 < \ell_2 \leq 2k}} \prod_{p \in P_{\ell_1, \ell_2 }}  p,
\end{equation}
for each $1 \leq \ell_1 < \ell_2 \leq 2k$, we may now define the singular series
\begin{equation}\label{singular}
\mathfrak{S}=\mathfrak{S}(W, \overline{ \mathcal{H} })=\prod_{p |Q } \left( 1-\frac{1}{p} \right)^{-1}.
\end{equation}
Finally, we choose a residue class $b \bmod W$ such that $(b+h_j,W)=1$ for each $j \leq k$ and set
$$a_h(n)= 1_{\mathcal{P}}( |n+h|), \qquad h \in \overline{\mathcal{H}}.$$ 
The key arguments in Sections \ref{GvsPsection} and \ref{normalisedsection} make use of the following observation: if we are able to produce a suitable weighted sum
\begin{align}\label{satz}
\sum_{n \leq N} \left(\sum_{i=1}^k a_{h_i}(n) -  (m-1) \right) w(n)^2 >0
\end{align}
then we may deduce the existence of an $m$-tuple $(n+h_{i_1},..., n+h_{i_m})$ consisting entirely of primes.\\ 
We will employ weight functions $w(n)$ of the shape
\begin{equation}\label{weight}
w_f(n)=\sum^{\star}_{ d_1,...,d_{2k}  }   \lambda_{\underline{d}}, 
\end{equation}
where the starred sum ranges over $2k$-tuples $\underline{d}=(d_1,...,d_{2k})$ such that $d_i |n+h_i$ and $(d_i,W)=1 $ for each index $i \leq 2k$. We set
\begin{align}\label{lambda}
\lambda_{\underline{d} }=\left(\prod_{i=1}^{2k} \mu(d_i) \right) f \left( \frac{ \log d_1}{\log R},...,\frac{ \log d_{2k}}{\log R}  \right), \qquad  R=N^{\delta}
\end{align}
for some smooth function $f:\left. \left[0, \infty \right)^{2k} \rightarrow \R \right.$ supported\footnote{Here we mean the restriction of any smooth $f:\R^{2k} \rightarrow \R $ such that $f(\underline{t})=0$ whenever $\underline{t} \in \left.\left[0, \infty \right) \right.^{2k}$ and $\sum_{i} t_i \geq 1$.} on the simplex $\Delta_{2k}=\Delta_{2k}(0,1)$ where, in general, we define 
$$\Delta_r(\eta, \tau) = \left\{t_1,...,t_r \geq \eta \left| \ \sum_{i=1}^r  t_i \leq \tau \right. \right\}$$
 for any pair of real numbers $0\leq  \eta< \tau  \leq 1$. The truncation parameter $R=N^{\delta}$, depends on the level of 
distribution of the primes. We will assume that  $\delta=1/4 + \epsilon_0$ for some arbitrarily small but fixed $\epsilon_0>0$. 
Finally, for each $l \leq 2k$ and each pair $1 \leq i < j \leq 2k$, we write
\begin{align}\label{Ddef}
Df&=\frac{\partial^{2k} f}{\partial t_1 ...\partial t_{2k} },  \qquad 
D_lf=\frac{\partial^{2k-1} f}{\partial t_1 ...\partial t_{l-1}\partial t_{l+1}...\partial t_{2k} } \notag \\
& \qquad D_{i,j}=\frac{\partial^{2k-2} f}{ \partial t_1 ...\partial t_{i-1}\partial t_{i+1}...\partial t_{j-1}\partial t_{j+1}...\partial t_{2k} }.
\end{align}

{\bf The partial Elliott-Halberstam conjecture EH[$\theta$].} In order to prove Theorem \ref{GvsP}, we will need to assume that the primes have level of distribution greater than $1/2$. To be precise, given any primitive residue class $a \bmod q$, define the discrepancy
\begin{equation}\label{disc}
\sum_{ m \leq x, m \equiv a(q) } 1_{\mathcal{P}}(m) =\frac{li(x)}{\varphi(q)}  + E(x,a,q), 
 \qquad li(x)=\int_{2}^{x} \frac{\ dt}{\log t}
\end{equation}  
for any $x>2$. Fixing $\theta \in (0,1)$, we will say that EH[$\theta$] holds if (see e.g.\cite[page 16]{Sound2})
\begin{equation}\label{BomVinplus}
\sum_{q \leq N^{\theta} }  \max_{a \in (\Z/q\Z)^{\times}} |E\left(N, a, q \right)| \ll_A \frac{N}{(\log N)^A} 
\end{equation}
for any $A>0$. In this language, the Bombieri-Vinogradov theorem is equivalent to the statement ``EH[$\theta$] holds for any $\theta \in (0,1/2) $''. 

The following lemma gives asymptotic estimates for the weighted sums appearing in \eqref{satz}. They are essentially proven in \cite[Sections 4 and 9]{Poly} but require some additional care in our current setting. We will give a proof of these results in the appendix. 

\begin{lemma}\label{mainasymp}
Assume that EH$[1/2+3\epsilon_0]$ holds with $\epsilon_0 \in (0,1/6)$. Then for any sufficiently large $k \in \N$ and any smooth function $f:\left. \left[0, \infty \right)^{2k} \rightarrow \R \right.$ supported on $\Delta_{2k}(0,1)$, one has the estimates
\begin{equation} \label{primenonprime}
 \sum_{n \leq N}^{'}   w_f(n)^2 \sim  \beta(N) I(f),
\qquad \qquad  \sum_{n \leq N}^{'}  a_{h_l}(n) w_f(n)^2 \sim \delta   \beta(N) J^{(l)}(f)
\end{equation}
for each index $1 \leq l \leq 2k$. The superscript $'$ indicates a restriction to natural numbers $n \equiv b \bmod W$ and we are also assuming that $\mathcal{P}(2k+1)|N$. $I,J$ are integrals given by 
\begin{align}\label{integrals}
 I(f)&= \int  Df(t_1,...t_{2k} )^2 \ dt_1...dt_{2k}, \notag\\
 J^{(l)}(f)&= \int  D_{l}f(t_1,...t_{l-1},0,t_{l+1},...t_{2k} )^2 \ dt_1...\widehat{ dt_{l} }...dt_{2k},
\end{align}  
where $\widehat{ dt_{l} }$ indicates an omission of the variable $t_l$. We have also used the shorthand
$$\beta(N)=\beta(N,W)=N \mathfrak{S} \frac{W^{2k-1} }{\varphi(W)^{2k}} ( \log R)^{-2k}.$$
\end{lemma}

\subsection{An autocorrelation estimate}
One of the key ingredients in the proof of Theorem \ref{GvsP} is a weighted auto-correlation estimate for the indicator $1_{\mathcal{P}}$. The following lemma is a simple modification of the bound obtained in \cite[Lemma 4.6]{Banks}.  Before proceeding we define, for each pair $1 \leq i < j \leq 2k$, the integral
\begin{equation*}
 L^{(i,j)}(f)= \int  D_{l}f(t_1,...t_{i-1},0,t_{i+1},...,t_{j-1},0,t_{j+1},...t_{2k} )^2 \ dt_1....\widehat{ dt_{i} } \widehat{ dt_{j} }...dt_{2k}.
\end{equation*}
\begin{lemma}[Auto-correlation estimate]\label{Mbound}
Assume that EH$[1/2+3\epsilon_0]$ holds with $\epsilon_0 \in (0,1/6)$. With notation as above, there exists a smooth function $f:\left. \left[0, \infty \right)^{2k} \rightarrow \R \right.$ supported on $\Delta_{2k}(0,1 )$  satisfying the estimates \footnote{Recall the notation $r(N,k)=o_k(1)$ when $\lim_{k \rightarrow \infty} r(N,k)=0$, independently of $N$.}
\begin{align}
&\sum_{n \leq N}^{'} a_{h_{i_0} }(n)  w_f(n)^2 \geq \delta \tilde{\psi}(2k)   \beta(N) I(f) (1+o_k(1)), \label{oneprimesum}\\
&\sum_{n \leq N}^{'} a_{h_i}(n) a_{h_j}(n)  w_f(n)^2 \leq \delta \tilde{\psi}(2k)^2    \beta(N) I(f) (1+o_k(1)) \label{twoprimesum}
\end{align}
for all $h_{i_0}$ and all pairs $h_i \neq h_j$ in $\overline{\mathcal{H}}$.  Here $\delta=1/4+\epsilon_0$.

\begin{proof}
We first set $\sigma=  2\psi(2k)$ and invoke \cite[Lemma 4.7]{Banks} to obtain a smooth function $f$ supported on $\Delta_{2k}(0,\sigma)$ such that 
$$J^{(i_0)}(f)\geq \tilde{\psi}(2k) (1+O( (\log k)^{-1/2})) I(f).$$ 
Combining this last estimate with Lemma \ref{mainasymp}, we immediately get \eqref{oneprimesum}.\\  
As for the proof of  \eqref{twoprimesum}, let us proceed as in \cite[Lemma 4.6 iii)]{Banks}. We set $i=1$ and $j=2k$ for ease of notation and aim to replace the indicator $1_{\mathcal{P}}(n+h_1)$ with a suitable sieve weight $( \sum_{d |n+h_1 }   u (\log d/\log R) )^2 $. To this end, let $u(y)$ be a smooth, real-valued function supported on the interval $[0,1-\sigma]$ satisfying the constraint $ u(0)=1$. In particular, we may choose $u$ to be a smooth approximation of the function $1-y/(1-\sigma)=1-y/(1- 2\psi(2k))$ supported on $[0,1-\sigma]$, yielding 
\[\mathcal{U}= \int [u'(y)]^2 \ dy = 1+O(\psi(k)).\]
Next we define $\tilde{f}(y,t_2,...,t_{2k})=u(y)\cdot f(0,t_2,...,t_{2k})$ and observe that $ \tilde{f}$ is supported on $\Delta_{2k}(0,1)$. Moreover, by viewing $\tilde{f} $ as a function of the $2k$ variables $y,t_2,...t_{2k}$, we have that 
$$D_{2k} \tilde{f}(y,t_2,...,t_{2k-1},0)=u'(y)D_{1,2k}(f)(0,t_2,...,t_{2k-1},0).$$

Using the pointwise inequality

\begin{equation*}
 1_{\mathcal{P}}(n+h_1) w_f(n)^2 \leq  \Bigg(  \sum_{ d \leq R, d |n+h_1  } u \left( \frac{\log d}{\log R} \right) \Bigg)^2  \left( \sum_{ \underline{d},  d_1=1} \lambda_{\underline{d}} \right)^2,
\end{equation*}
we may apply Lemma \ref{mainasymp} to find that
\begin{align*}
\sum_{n \leq N}^{'} 1_{\mathcal{P}}(n+h_1) 1_{\mathcal{P}}(|n+h_{2k}|)   w_f(n)^2 &\leq \sum_{n \leq N}^{'} 1_{\mathcal{P}}(|n+h_{2k}|)   w_{\tilde{f}}(n)^2 \sim \delta  \mathcal{U} \beta(N) L^{1,2k}(f). 
\end{align*}
To conclude the proof of \eqref{twoprimesum} we invoke the bound $L^{1,2k}(f)\leq \tilde{\psi}(2k)^2 (1+o_k(1)) I(f)$ from \cite[Lemma 4.7]{Banks}.
It should be noted that the integrals in \cite{Banks} are expressed in terms of the function $F=D f$ (recall \eqref{Ddef}). By the fundamental theorem of calculus, their integrals correspond exactly to our own definitions of $I(f),J(f), L(f)$. 

\end{proof}
 \end{lemma}

\section{Goldbach versus de Polignac numbers}\label{GvsPsection}
In this section we prove Theorem \ref{GvsP} by combining the sieve estimates of Lemmas \ref{mainasymp} and \ref{Mbound} with the K\H{o}v\'ari--S\'os--Tur\'an Theorem from graph theory. We shall assume that the conjecture EH$[1/2+3\epsilon_0]$ holds for some small but fixed $\epsilon_0>0$.\\
Let $k$ be a large natural number and suppose $\mathcal{H}= \left\{h_{1}, ...,h_{k} \right\}$ is an admissible $k$-tuple with each shift $h_j$ satisfying $\mathcal{P}(2k+1) | h_j$. Assume furthermore that $N \equiv 0 \ (\mathcal{P}(2k+1))$ is sufficiently large with respect to $k$. Associated to $\mathcal{H}$ there is another tuple $\mathcal{H}'= \left\{h_{k+1}, ...,h_{2k} \right\}$, where $h_j=h_{j-k}-2N$ for each $j=k+1,...,2k$. Setting $\overline{ \mathcal{H} }:= \mathcal{H} \cup \mathcal{H}'$, we recall the notation
$$a_h(n)= 1_{\mathcal{P}}( |n+h|)=
\left\{
	\begin{array}{ll}
		1_{\mathcal{P}}(n+h)  & \mbox{if } h \in \mathcal{H} \\
		1_{\mathcal{P}}(-h- n) & \mbox{if } h \in \mathcal{H}'
	\end{array}
\right.$$
and take $w=w_f(\overline{ \mathcal{H} })$ to be the weight function formed with the smooth function $f$ from Lemma \ref{Mbound}. It follows that
\begin{equation}\label{Tasymp}
\mathcal{T}(\overline{ \mathcal{H} })=\sum_{n \leq N}^{'} \left(\sum_{h \in \overline{ \mathcal{H} } } a_h(n) \right) w(n)^2
\geq (2k \delta \beta(N) \tilde{\psi}(2k) I_{2k}(f) ) (1+o_k(1)).
\end{equation}
Now write $X(n)=X_{\overline{ \mathcal{H} }}(n)=\sum_{h \in \overline{ \mathcal{H} } } a_h(n)$ and apply Cauchy-Schwarz to get the upper bound
\begin{align}\label{CS}
\mathcal{T}(\overline{ \mathcal{H} }) & \leq \left(\sum_{n \leq N}^{'} 1_{ \left\{ X >0 \right\}   } (n) w(n)^2 \right)^{1/2} \left( \sum_{n \leq N}^{'} \sum_{  h,\tilde{h} \in \overline{ \mathcal{H} }} a_h(n) a_{\tilde{h}}(n) w(n)^2  \right)^{1/2} \notag \\
& \leq \left(\sum_{n \leq N}^{'} w(n)^2 \right)^{1/2} \Big( \sum_{n \leq N}^{'} \sum_{ \substack{ h,\tilde{h} \in \overline{ \mathcal{H} }\\ h \neq \tilde{h}  }} a_h(n) a_{\tilde{h}}(n) w(n)^2 + \mathcal{T}(\overline{ \mathcal{H} }) \Big)^{1/2}.
\end{align}
Combining \eqref{Tasymp}, \eqref{CS} and Lemma \ref{mainasymp}, we get the lower bound 
\begin{align}\label{Tupper}
\sum_{n \leq N}^{'} \sum_{ \substack{ h,\tilde{h} \in \overline{ \mathcal{H} }\\ h \neq \tilde{h}  }} a_h(n) a_{\tilde{h}}(n) w(n)^2 
& \geq  \mathcal{T}(\overline{ \mathcal{H} }) \left( 
\frac{\mathcal{T}(\overline{ \mathcal{H}})  }{ \sum_{n \leq N}^{'}  w(n)^2  } -1 \right) \\
& \geq  \frac{ \left( 2k \delta \beta(N) \tilde{\psi}(2k) I_{2k}(f)  \right)^2 }{\beta(N)  I_{2k}(f)  }   (1+o_k(1)).  \notag
\end{align}
We may now rearrange the double sum on the LHS of the previous inequality. Let $\mathcal{M}:= (\mathcal{H} \times \mathcal{H}) \cup (\mathcal{H}' \times \mathcal{H}')$ and write $\mathcal{M}^c$ for the complement of $\mathcal{M}$ in $\overline{ \mathcal{H} } \times \overline{ \mathcal{H} }$. We gather that
\begin{align}\label{goldbachpairs}
\sum_{n \leq N}^{`} \sum_{ \substack{ (h,\tilde{h}) \in \mathcal{M}^c \\ h \neq \tilde{h}  }} a_h(n) a_{\tilde{h}}(n) w(n)^2&
\geq   4k^2 \delta^2  \tilde{\psi}(2k)^2 \beta(N) I(f)  (1+o_k(1)) \\ \notag
&- \sum_{n \asymp N}^{`} \sum_{ \substack{ (h,\tilde{h}) \in \mathcal{M} \\ h \neq \tilde{h}  }} a_h(n) a_{\tilde{h}}(n) w(n)^2. 
\end{align} 
We now recall that $\delta=1/4+\epsilon_0$. Let us choose a small $\epsilon_1>0$ such that $(1+ 2 \epsilon_1)<4\delta$ and consider two mutually exclusive assumptions. \\
\\
{\bf Hypothesis A} We say hypothesis A holds if there exists an increasing sequence of natural numbers $k$ satisfying the following condition. For each admissible $k$-tuple $\mathcal{H}$ at least $1/2 - \epsilon_1$ of all pairs $1 \leq i <j \leq k$ produce a difference $h_j-h_i$ which is not a de Polignac number.\\
First suppose that hypothesis A is true and let $n \in [N/2,N]$. It follows that $a_h(n) a_{\tilde{h}}(n)=0$ for at least $1/2 -\epsilon_1$ of all pairs $(h,\tilde{h}) \in \mathcal{M} $. Plugging this information back into (\ref{goldbachpairs}) and applying Lemma \ref{Mbound}, we find that
\begin{align*}
\sum_{n \leq N} \sum_{ \substack{ (h,\tilde{h}) \in \mathcal{M}^c \\ h \neq \tilde{h}  }} a_h(n) a_{\tilde{h}}(n) w(n)^2&
\geq  4k^2 \delta^2  \tilde{\psi}(2k)^2 \beta(N) I(f)   (1+o_k(1)) \\
&-  (1/2 + \epsilon_1)\delta  \tilde{\psi}(2k)^2 \beta(N) I(f) \frac{2k(2k-1)}{2}   (1+o_k(1)).
\end{align*}    
We see that the RHS is a positive quantity for $N$ and $k$ sufficiently large. From this we deduce the existence of some $n\leq N$ and a pair $h_i,h_j \in \mathcal{H}$ for which $n+h_i$ and $2N-n-h_j$ are both prime. This implies that all sufficiently large $2N$ lie within a bounded distance from a Goldbach number. \\
\\
Now consider the case where hypothesis A fails and write $\mathcal{D}$ for the set of de Polignac numbers. Let $k$ be any sufficiently large number and $\mathcal{H}$ an admissible $k$-tuple. Then at least $1/2 + \epsilon_1$ of all pairs $1 \leq i <j \leq k$ give a difference $h_j-h_i$ which is a de Polignac number. As a straightforward consequence we get the following useful property. Let $U:=\left\{ u_1,...,u_k\right\}\subset 2\N $ and $V:=\left\{ v_1,...,v_k\right\}\subset 2\N $ be a pair of sets for which $U \cap V = \emptyset$ and $U \cup V$ is admissible. Then there exists a $(u,v) \in U \times V$ with $|u-v| \in \mathcal{D}$. We will say $(\mathcal{D},k)$ satisfies the \textsl{cross product property}. To finish the proof of Theorem \ref{GvsP} we need the following lemma which was proven in a private communication with S.Miner and S. Das. 
\begin{lemma}\label{CPP}
Let $k \in \N$ be arbitrary and suppose $(\mathcal{D},k)$ satisfies the cross product property. Then $\mathcal{D}$ has full asymptotic density in $2 \N$. Moreover, we have the power saving
\begin{align}\label{Ddensity}
\left| \mathcal{D}^c \cap 2[N]\right| \ll N^{\kappa}
\end{align}
for some $\kappa <1$ depending on $k$.
\end{lemma}

\begin{remark} There is an expedient way of establishing the full density of $\mathcal{D}$ without the power saving result. Indeed, suppose for  contradiction that the set $A:=\mathcal{D}^{c} \cap 2 \N$ has positive upper density and let $\mathcal{P}(y)=\prod_{p \leq y} p$. An application of Szemer\'edi's Theorem \cite{Sze} gives a $(2k-2)$-term arithmetic progression $P=(b+ra)_{r \leq 2k-2} \subset A$. We may assume without loss of generality that $a \equiv 0 \bmod \mathcal{P}(2k)$. Now consider the pair $U=\left\{ a,2a...,ka\right\}$ and $V=\left\{ b+ka,b+(k+1)a...,b+(2k-1)a\right\}$. Clearly $U$ and $V$ do not intersect and their union is admissible. Since the difference set $\left|U-V\right|=P \subset A$, the cross product property gives the desired contradiction. \end{remark}
 
We now turn to the estimate for $\mathcal{D}^c \cap [N]$. The result will follow from two simple lemmas. Call a pair $\{x,y\}$ an $A$-pair if $\left|y-x\right| \in A$. 
\begin{lemma} \label{lem:surfing}
For every $k$ there is an $\ell = \ell(k)$ such that if $U$ and $T$ are two disjoint subsets of $2[N]$ of size $\ell$, then there are $X \subset U$ and $Y \subset T$ such that $\left|X\right| = \left|Y\right| = k$ and $X \cup Y$ is admissible.
\end{lemma}

Given the above lemma, we shall use the classic result of K\H{o}v\'ari--S\'os--Tur\'an~\cite{kst} on the Tur\'an number of complete bipartite graphs to resolve the problem.

\begin{theorem}[K\H{o}v\'ari--S\'os--Tur\'an, 1954] \label{thm:kst}
If $G$ is a graph on $n$ vertices that does not contain $K_{t,t}$ as a subgraph, then $G$ has at most $c t^{1/t} n^{2 - 1/t} + O(n)$ edges.
\end{theorem}

\begin{proof}[ Proof of Lemma \ref{CPP}]
Let $H$ be a graph with vertices $V = 2[N]$, and edges
$$ E = \left\{ \{x,y\} : \{x, y\} \textrm{ is \emph{not} a $\mathcal{D}$-pair} \right\}. $$
Let $(\mathcal{D},k)$ satisfy the cross product property and let $\ell = \ell(k)$ be as in Lemma \ref{lem:surfing}.  We claim that $H$ is $K_{\ell,\ell}$-free.  Indeed, suppose for contradiction $K_{\ell,\ell} \subset H$, and let $U$ and $T$ be the two vertex sets on which this copy of $K_{\ell,\ell}$ is realised.  In particular, we must have $U \times T \subset E(H)$, and so there are no $\mathcal{D}$-pairs in $U \times T$.

However, by Lemma \ref{lem:surfing}, we can find two $k$-sets $X \subset U$ and $Y \subset T$ such that $X \cup Y$ is admissible. By assumption there must be some $\mathcal{D}$-pair in $X \times Y \subset U \times T$, giving the necessary contradiction.

Thus $H$ is indeed $K_{\ell,\ell}$-free, and by Theorem~\ref{thm:kst} has $O(N^{2 - 1/\ell})$ edges. Observe that there are $N-d$ edges corresponding to a difference of $2d$ and moreover, we need at least $\binom{t+1}{2}$ edges to cover $t$ distinct differences.  Since $H$ has only $O(N^{2 - 1/\ell})$ edges, it can cover at most $O(N^{1 - 1/(2\ell)})$ distinct differences, and thus we must have that $| \mathcal{D}^c \cap 2[N]| =O(N^{1 - 1/(2 \ell)})$.
\end{proof}

It remains to prove Lemma \ref{lem:surfing}.

\begin{proof}[Proof of Lemma \ref{lem:surfing}.]
Since we wish to find sets $X$ and $Y$ of size $k$, we need only consider the primes  $ p_1 < \hdots < p_m \le 2k$, where $m = \pi(2k)$. Recall that we are given two sets $U, T \subset 2[N]$. To begin, set $X_1 = U$ and $Y_1 = T$ and $\ell_0 = \ell = 3^{m} k$.

Now suppose that for $1 \le i \le m-1$ we are given subsets $X_i \subset U$ and $Y_i \subset T$, both of size $\ell_i$, such that $X_i \cup Y_i$ does not occupy all residue classes modulo $p_j$ for any $1 \le j \le i$.  By the pigeonhole principle, there is some residue class $C$ modulo $p_{i+1}$ such that $\left|(X_i \cup Y_i) \cap C \right| \le 2 \ell_i / p_{i+1}$.  Let $X_{i+1}' = X_i \setminus C$ and $Y_{i+1}' = Y_i \setminus C$.  Let $\ell_{i+1} = \ell_i (1 - 2/p_{i+1})$, and observe that this gives a lower bound on the sizes of $X_{i+1}'$ and $Y_{i+1}'$.  Finally, take $X_{i+1}$ and $Y_{i+1}$ to be arbitrary subsets of $X_{i+1}'$ and $Y_{i+1}'$ of size $\ell_{i+1}$, and note that these sets do not occupy the residue class $C$ modulo $p_{i+1}$.

Repeating these process, we arrive at sets $X_m \subset U$ and $Y_m \subset T$ of size $\ell_m$ that do not occupy all residue classes modulo $p_j$ for any $1 \le j \le m$, and hence are admissible.  All that remains is to verify the lower bound $\ell_m \ge k$, after which we may take $X$ and $Y$ to be arbitrary $k$-subsets of $X_m$ and $Y_m$.
By construction, we have $\ell_m = \ell_{m-1} \left( 1 - 2/p_m \right) = \hdots = \ell \prod_{j=2}^{m} \left( 1 - 2/p_j \right) \ge \ell / 3^m = k$, completing the proof. 
\end{proof}

\section{A note on the sequence of normalised prime gaps}\label{normalisedsection}

In this section we give a proof of Proposition \ref{Ldensity}. The argument relies on a conjectural form of \cite[Theorem 4.2]{Banks} together with a density estimate for sum-free sets and a simple Cauchy Schwarz estimate. We first require some notation and background results.
\subsection{Preliminaries}
We begin by recalling \cite[Lemma 4.1]{Banks}.

\begin{lemma}\label{exceptional}
There exists an absolute constant $c>0$ with the following property. Given any parameter $T\geq 3$ and $P= T^{1/ \log_2 T}$, there is at most one modulus $q$ such that $q \leq T$ and $P^{+}(q)\leq P$ and at most one primitive character
$\chi \bmod q$ for which $L(s,\chi)$ has a zero in the region
$$Re(s)\geq 1-\frac{c}{\log P}, \qquad \left|Im(s)\right| \leq \exp[\log P/(\log T)^{1/2}]. $$
In this case, one has the bounds
$$P^{+}(q)\gg \log q \gg \log_2 T.$$
\end{lemma}
Following the notation of Lemma \ref{exceptional} we introduce the quantities
$$Z(T)=P^{+}(q), \qquad  \ w= \epsilon \log N, \qquad  W=\prod_{\substack{p \leq w \\ p \nmid Z(N^{4 \epsilon})}} p,$$
and set $Z(T)=1$ if the exceptional modulus $q\leq T$ of the lemma does not exist.\\ 
{\bf A modified Elliott-Halberstam conjecture $EH^{*}[\theta]$.} Given $\theta \in (0,1)$ we will say that $EH^{*}[\theta]$ holds if there is an $\epsilon'=\epsilon'(\theta)$ with the following property: for any $\epsilon \in (0,\epsilon'), c \in (0, \theta) $ and any squarefree $q_0$ satisfying $P^{+}(q_0) \leq N^{\epsilon/\log_2 N}$, one has the estimate
\begin{align}\label{bomvin}
\sum_{\substack{ q \leq N^{\theta - c}\\ \substack{ q_0|q \\(q,Z(N^{2\epsilon}) )=1} }} \max_{(a,q)=1} \left|\psi(N;q,a) -\frac{N}{\varphi(q)}\right| \ll_{A,c} \frac{N}{\varphi(q_0) (\log N)^A}.
\end{align}

In \cite[Theorem 4.2]{Banks} it was demonstrated  that (\ref{bomvin}) holds with $\theta=1/2 $.\\
{\bf Key Assumption.} For the remainder of Section \ref{normalisedsection} we will assume that $EH^{*}[\theta$] holds for some $\theta \in (2/3,1)$.\\

Throughout this entire section we will take $k \in \N$ growing slowly to infinity with $N$ and let $\mathcal{H}=\left\{ h_1,...,h_k\right\}$ be an admissible $k$-tuple for which each member is bounded in size by $N$. Assume also that each prime dividing $\prod_{1 \leq i <j \leq k} (h_i-h_j)$ is smaller than $w$.\\
In order to prove Proposition \ref{Ldensity} we require a modified version of the weight (\ref{weight}) in which the 
$$\lambda_{\underline{d}}=\left( \prod_{i=1}^k \mu(d_i) \right)  f \left( \frac{\log d_1}{\log N},...,\frac{\log d_k}{\log N} \right)$$   
are supported on $k$-tuples for which $((\prod_{i=1}^k d_i), Z(N^{4 \epsilon}))=1$. We let $\nu$ denote the associated weight function given in (\ref{weight}) and we choose a residue class $b \bmod W$ such that $(b+h_j,W)=1$ for each $j \leq k$.  As in section \ref{sievesection}, we let $\psi: \N \rightarrow (0,\infty)$ tend to zero in such a way that $\psi(k) \log k \rightarrow \infty$ and write $\tilde{\psi}(k)=\psi(k) \log k/k$. By choosing the functions $f$ appropriately, the following bounds were proven in \cite[Lemma 4.6  (i)-(iii) and Lemma 4.7]{Banks}:
\begin{align*}
\text{a)} & \ \sum_{n \leq N}^{`} \nu(n)^2 \sim \rho(N) I_k(F) \\ 
\text{b)} & \ \sum_{n \leq N}^{`} 1_{\mathcal{P}}(n+h) \nu(n)^2  \geq  \rho(N) \tilde{\psi}(k) I_k(F) (1+o(1))\\ 
\text{c)} & \ \sum_{n \leq N}^{`} 1_{\mathcal{P}}(n+h)1_{\mathcal{P}}(n+h') \nu(n)^2   \leq \frac{2}{\theta}  \rho(N) \tilde{\psi}(k)^2 I_k(F) (1+o(1)).
\end{align*}
As before, the superscript $'$ indicates a restriction to natural numbers $n \equiv b(W)$. We have also written $F=Df=\frac{\partial^{k} }{\partial t_1 ...\partial t_{k}} f$ and 
$$\rho(N)=\rho(N,W)=N \frac{W^{k-1} }{\varphi(W)^{k}} ( \log N)^{-k}.$$

\begin{remark}
The appearance of the factor $2/\theta$ in equation c) just above comes from a straightforward modification of the arguments in \cite{Banks}. Indeed, thanks to the assumption $EH^{*}[\theta]$, we may take the function $G$, in the proof of \cite[Lemma 4.6 iii)]{Banks}, to be supported on the interval $[0, \theta/2- \psi(k)]$ rather than the interval $[0,1/4- \psi(k)]$. Moreover, the very same assumption  $EH^{*}[\theta]$ allows one to control the error terms in b) and c), precisely as in the proof of \cite[Lemma 4.6 ii)]{Banks}.    
\end{remark}

\begin{lemma}\label{cells}
Let $k \in \N$ be a large natural multiple of $3$. Suppose, in addition to all of the above assumptions, that the $k$-tuple $\mathcal{H}$ is partitioned into equally sized cells
$$\mathcal{H}=\mathcal{H}_{1} \cup \mathcal{H}_2 \cup \mathcal{H}_{3}$$
so that $|\mathcal{H}_{i} |=k/3 $ for each $i$. Then there exists an $n_1 \in [N,2N]$ for which $n_1 \equiv b \bmod W$ together with a pair of distinct indices $ i_{1}, i_{2} \in \left\{1,2,3\right\}$  satisfying
$$|\mathcal{H}_{i}(n) \cap \mathcal{P}|\geq 1 \text{ for } i= i_{1},i_{2}. $$
\begin{proof}
Write $Y_j(n)=\sum_{h \in \mathcal{H}_j} 1_{\mathcal{P}}(n+h)$ and introduce the expressions 
$$\mathcal{U}_j(\mathcal{H})= \sum_{n \leq N}^{'} Y_j(n) \nu(n)^2, \qquad 1 \leq j \leq 3.$$
Observe that the lemma will follow if we are able to show that the weighted sum 
\begin{align*}
A=&\sum_{n \leq N}^{'} \left(\sum_{j=1}^{3} 1_{\left\{Y_j \geq 1\right\}}(n) -1 \right) \nu(n)^2
\end{align*} 
is strictly positive. Let us first record a rough upper bound for $\mathcal{U}_j(\mathcal{H})$. Combining the estimates a)-c) with the bound in the first line of \eqref{Tupper} (replacing $\mathcal{T}$ with $\mathcal{U}_j$), we readily find that 
\begin{align*}
\mathcal{U}_j(\mathcal{H}) &\ll 
\Big( \sum_{n \leq N}^{'} \sum_{ \substack{h,h' \in \mathcal{H}_j\\ h\neq h'} } 1_{\mathcal{P}}(n+h) 1_{\mathcal{P}}(n+h') \nu(n)^2 \Big)^{1/2}
\left(\sum_{n \leq N}^{'} \nu(n)^2 \right)^{1/2} \\
& \ll \rho(N)   \tilde{\psi}(k) k I_k(F).
\end{align*}
On the other hand, invoking the estimate c) once again, we also have that 

\begin{align*}
\sum_{n \leq N}^{'} &Y_{j}^2(n)\nu(n)^2= \sum_{n \leq N}^{'} \sum_{ \substack{h,h' \in \mathcal{H}_j\\ h\neq h'} } 1_{\mathcal{P}}(n+h) 1_{\mathcal{P}}(n+h') \nu(n)^2
+ \sum_{n \leq N}^{'}  \sum_{h \in \mathcal{H}_j } 1_{\mathcal{P}}( n+h) \nu(n)^2\\
&\leq  \rho(N) I_k(F) (1+o(1))  \left(\frac{2}{\theta}   \tilde{\psi}(k)^2  (k/3)^2 + O \left( \tilde{\psi}(k)k \right) \right) \\
&= \rho(N) I_k(F) (1+o(1))\left(\frac{2}{\theta}   \frac{  \tilde{\psi}(k)^2  k^2}{9}  \right).
\end{align*}
It now follows that

\begin{align*}
\sum_{n \leq N}^{`} 1_{ \left\{ Y_j \geq 1 \right\}   } (n) \nu(n)^2& \geq \Bigg( \sum_{n \leq N}^{`}  \sum_{h \in \mathcal{H}_j } 1_{\mathcal{P}}( n+h) \nu(n)^2 \Bigg)^2 \Bigg(\sum_{n \leq N}^{`} Y_{j}^2(n)\nu(n)^2 \Bigg)^{-1} \\
& \geq  \Big[   \rho(N) \frac{\tilde{\psi}(k) k}{3} \ I_k(F) \Big]^2   \Big[   \rho(N) I_k(F)  \left(\frac{2}{\theta}   \frac{  \tilde{\psi}(k)^2  k^2}{9}  \right)   \Big]^{-1} (1+o(1)) \\
& \geq \rho(N) I_k(F) \frac{\theta}{2}(1+o(1)).
\end{align*}

Plugging this last estimate into $A$, we get
\begin{align*}
A& \geq   \rho(N) I_k(F) (1+o(1)) \left(  \frac{3\theta}{2}  -1  \right)>0,
\end{align*}
which concludes the proof.
\end{proof}
\end{lemma}

\subsection{Sum-free sets in $\R^{+}$}

We now turn to the proof of Proposition \ref{sumfree}. Recall that our goal is to establish the density estimate
\[ \overline{d}(A) \leq 1/3\]
for any measurable sum-free set $A \subset \R^{+}$. 
 Let $m$ denote the Lebesgue measure on $\R$ and write $A_{+}:=A \cap \R^{+}$ for any $A \subset \R$. Recall that a set $A \subset \R^{+}$ is said to be sum-free if the equation $x+y=z$ has no solutions in $A^3$; the asymptotic upper density of $A$ is given by
$$\overline{d}(A):= \limsup_{T \rightarrow \infty } \frac{m(A \cap [0,T])}{T}.$$
Given two measurable sets $A, B$ denote the relative density $d(A;B):= m(A \cap B)/m(B)$ when $m(B)$ is finite and non-zero. We write $A^c$ for the complement of $A$ in $\R^{+}$.\\
Before proceeding with the proof of the density estimate, we recall two well-known theorems, due to Steinhaus and Lebesgue respectively. First, when $m(A)>0$ there exists an open interval $I\ni 0$ contained in $A-A$. Second, the Lebesgue density Theorem asserts that  
$$\lim_{x \rightarrow 0} m(A \cap [y-x,y+x])/2x=1 \text{ for almost all } y \in A.$$

{\bf Proof of Proposition \ref{sumfree}}.
Cleary we may assume that $E$ is unbounded and $m(E)>0$. As mentioned just above, there exists an interval of the form 
$$I=(0,s) \subset (E-E)_{+}.$$
We may thus attach to each $y \in I$ the non-empty set $\mathcal{S}(y):=\left\{ x \in E | \ x+y \in E \right\} $. To prove the result we make the following claim. \\

{\bf Claim.} Given any natural number $M \geq 1$, there exists an evenly spaced $M$-tuple $(y_i)_{i=1}^{M}=(iy)_{i=1}^{M} \subset I$ for which 
$$(E-E)_{+} \supsetneq \left[E- \cup_{i=1}^M\mathcal{S}(y_i) \right]_{+}.$$
Suppose the claim is true and fix any $M \geq 1$. We obtain a pair $e_1< e_2$ in $E$ so that $e_2-e_1 \notin \left[E- \cup_{i=1}^M\mathcal{S}(y_i) \right]_{+}$. Recalling that $(E-E)_{+} \subset E^{c}$, one finds the sets
$$E, E+e_1, \cup_{i=1}^M\mathcal{S}(y_i) +e_2 $$
 to be pairwise disjoint. As a result one has that
$$3 \cdot \overline{d} \left( \cup_{i=1}^M\mathcal{S}(y_i) \right) \leq 2 \cdot \overline{d} (E) +  \overline{d} \left( \cup_{i=1}^M\mathcal{S}(y_i) \right)  
= \overline{d} \left(  E \cup \left[ E+e_1 \right] \cup \left[ \cup_{i=1}^M\mathcal{S}(y_i) +e_2\right]  \right) \leq 1. $$

This yields the estimate $\overline{d} \left( \cup_{i=1}^M\mathcal{S}(y_i) \right) \leq 1/3$. Our aim is to estimate $\overline{d}(E)$ via the decomposition $E=[\cup_{i=1}^M\mathcal{S}(y_i)] \cup [\left( \cup_{i=1}^M\mathcal{S}(y_i) \right)^c \cap E]$. Since, by construction, the sets
$$[\left( \cup_{i=1}^M\mathcal{S}(y_i) \right)^c \cap E]+y_j \quad j=1, \dots, M$$
are pairwise disjoint, we gather that $\overline{d} \left( \left( \cup_{i=1}^M\mathcal{S}(y_i) \right)^c \cap E \right) \leq 1/M$. Seeing that $M$ may be chosen arbitrarily large, we get $\overline{d}(E) \leq 1/3$, as desired. It remains to prove the claim.\\

{\bf Proof of claim.} Let us assume for contradiction that the claim is false. Then we obtain a value $M \in \N$ such that 
\begin{align}\label{contradiction}
(E-E)_{+}=\left[E- \cup_{i=1}^M\mathcal{S}(y_i) \right]_{+}
\end{align}	
for all evenly spaced $M$-tuples $(y_i)_{i=1}^{M}=(iy)_{i=1}^{M} \subset I$. From this statement we will derive two consequences.
\begin{itemize}
\item[(i)]  The difference set $(E-E)_{+}$ is dense in $\R^{+}$.
\item[(ii)] For any open interval $J \subset \R^{+}$ satisfying $m(J)\leq s$, one has that  $d(E^c; J) \gg 1/M^2 $.
\end{itemize}

To prove (i), first note that $(E-E)_{+}$ is unbounded (since $E$ is). Therefore it is enough to show that $(E-E)_{+}$ is dense in $(0,z)$ for any $z \in (E-E)_{+}$. In fact it will suffice to show that  $(E-E)_{+}$ is dense in $(z-s,z)$ since we may then find a $z' \in [(E-E) \cap (z-s, z-s/2)]_{+}$ and repeat the argument on the interval $(z'-s, z')$. Continuing in this manner one arrives at the origin in finitely many steps.\\
Now fix $z \in (E-E)_{+}$ and let $t \in (z-s,z)$ be arbitrary. In order to construct a sequence $(z_n)_n \subset (z-s,z)\cap (E-E)_+$ which converges to $t$, we proceed as follows. Define the evenly spaced sequence $(iy_1)_{i=1}^{M}$ with $y_1=(z-t)/M$. An application of (\ref{contradiction}) guarantees some index $i_{1}$ for which $z_1:=z-i_1 y_1 \in E-E$ and we get that $|z_1 -t|\leq  (1-1/M) |z -t|$. This process may be reiterated on the interval $(t,z_1)$ to obtain an element $z_2 \in E-E$ satisfying $|z_2 -t| \leq  (1-1/M) |z_1 -t|\leq  (1-1/M)^2 |z -t|$. By induction we get the desired sequence $z_n \rightarrow t$.\\
To prove (ii), consider any interval $J=(a,b) \subset \R^{+}$ of smaller length than $I$. Invoking property (i), we may assume for notational convenience that $b \in E-E $. Now write $J_M:= \left(0, (b-a)/M \right) $ and introduce the collection
$$\mathcal{A}_j:=\left\{x \in J_M| \ b- jx \in E-E \right\} \quad  j=1,\dots,M.$$
We will show that the $\mathcal{A}_j$'s cover the set $J_M$. Indeed, let $x \in J_M$ be arbitrary and consider the string $(x_i)_{i=1}^M=(ix)_{i=1}^M$. Applying (\ref{contradiction}), we find an index $j \leq M$ such that $b=e-e'$ with $e \in E$ and $e' \in \mathcal{S}(x_j)$ and hence we gather that $b-jx=e-(e'+jx) \in E-E$. It follows that $J_M=\cup_{j=1}^{M} \mathcal{A}_{j} $ and this in turn implies the lower bound $d(\mathcal{A}_i; J_M)\gg 1/M $ for some index $i$. Consequently one finds that 
$$ d(E^c; J) \geq d(E-E; J) \geq d(b- i \cdot \mathcal{A}_i ; J)\gg 1/M^2, $$
yielding property ii). To conclude the proof of the claim we observe that (ii) holds uniformly over all small intervals $J$. By the Lebesgue density Theorem this is only possible if $m(E)=0$  which gives the desired contradiction.\hfill$\square$

\begin{remark}
The Proof of Proposition \ref{sumfree} was inspired by the work of \L uczak \cite{Luc}, which deals with sum-free sets in discrete settings.
\end{remark}

\subsection{Concluding the proof of Proposition \ref{Ldensity}}
From this point onwards the proof of Proposition \ref{Ldensity} is carried out as in \cite[Section 6]{Banks}. We give a brief summary of the argument for the reader's convenience. Suppose $k$ is a large positive integer multiple of $3$. Let $\beta_{3}\geq ...\geq \beta_1>0$ be an arbitrary triple, fix any small $\epsilon \in (0,1)$ and set 
$$x = \epsilon^{-1}, \qquad y=w=\epsilon \log N, \qquad z=y (\log_2 y)(\log_3 y)^{-1}.$$
The discussion in \cite[Section 6]{Banks} (leading up to eq. (6.6)) yields the existence of a $k$-tuple $\mathcal{H}=\mathcal{H}_{1} \cup \mathcal{H}_2 \cup \mathcal{H}_{3}$ for which each set in the partition is of size $k/3$ and 
$$h_j=(\beta_j + \epsilon +o(1))\log N \text{ for all} \ h_j\in \mathcal{H}_j.$$
Moreover, there exists a suitable residue class $b \bmod W$ and an integer $n>y$ such that $n \equiv b \bmod W$ and $[n,n+z] \cap \mathcal{P} = \mathcal{H}(n) \cap \mathcal{P}$. We gather that the primes in $\mathcal{H}(n)$ are consecutive. By Lemma \ref{cells} there are at least two indices $1 \leq i<j \leq 3$, coming from distinct cells in the partition of $\mathcal{H}$, yielding primes with difference
$$\frac{p_{r+1}-p_{r} }{\log p_r} = \beta_{j} - \beta_{i} +o(1).$$
Since this construction holds for infinitely many values of $N$, we easily deduce the following property. 
\begin{align}\label{betacondition}
&\text{For any string } 0\leq \beta_1<\beta_2< \beta_{3} \text{ there exists a pair }  i <j
 \text{ satisfying } \beta_{j}-\beta_{i} \in \mathcal{L}.
\end{align}

Property (\ref{betacondition}) just above implies that the set $\mathcal{L}^c$ is sum-free. Indeed, suppose that we are given an arbitrary pair $0 \leq \alpha_1 < \alpha_2$ belonging to $\mathcal{L}^c$. Let $0 \leq \beta_1 <\beta_2 < \beta_3$ be chosen to satisfy $\beta_2- \beta_1= \alpha_1$ and $\beta_3-\beta_2=\alpha_2$. By \eqref{betacondition} we have that $\alpha_1 + \alpha_2=\beta_3-\beta_1 \in \mathcal{L}$, as desired. We may thus apply Proposition \ref{sumfree} to get the lower bound \eqref{lowerdensityL}.

 \renewcommand{\theequation}{A-\arabic{equation}}
 \renewcommand{\thetheorem}{A.\arabic{theorem}}
  \setcounter{equation}{0}  
  \setcounter{theorem}{0} 
  \section*{APPENDIX}  
\subsection*{Evaluating the weighted sum} In this final section we will discuss Lemma \ref{mainasymp}. The proof follows that of \cite[Lemma 4.1]{Poly} very closely, although there are some important technical differences, owing to the possible presence of large primes in $P_{\ell_1,\ell_2}$ (see \eqref{Pl1l2}). Recall that $R=N^{\delta}$ with $\delta=1/4+\epsilon_0$ for some small $\epsilon_0>0$. Given any tuple $\underline{t} \in  \R^{2k}$ we will use the shorthand $f_R(\underline{t})=f(\log t_1/\log R,..,\log t_{2k}/\log R)$. Let us consider the first sum in \eqref{primenonprime}. Expanding $\sum_{n \leq N}^{'} w_f(n)^2 $, we get 
\begin{align}\label{satz1}
&\sum_{\underline{d}, \underline{e}}^{\star} \left(\prod_{i=1}^{2k} \mu(d_i)  \mu(e_i)\right) 
f_R \left(\underline{d} \right) f_R\left( \underline{e}\right) 
\sum_{ \substack{n \leq N\\ n \equiv -h_i ([d_i,e_i]) \\ n\equiv b(W)   }} 1 \notag \\ 
&= \sum_{\underline{d}, \underline{e}}^{\dagger} \left(\prod_{i=1}^{2k} \mu(d_i)  \mu(e_i)\right) 
f_R \left(\underline{d} \right) f_R\left( \underline{e}\right) 
\left(\frac{N}{q(\underline{d}, \underline{e})} +O(1)\right).
\end{align}
Here we used the notation $q(\underline{d}, \underline{e}):=W \cdot lcm( [d_1,e_1],...[d_{2k},e_{2k}])$. The sum $\sum^{\dagger}$ ranges over vectors $(\underline{d},\underline{e})$  which are composed of squarefree integers satisfying the following conditions:

\begin{align}\label{i}
\begin{split}
& \text{ i) $\gcd([d_j,e_j],W)=1$ for each $j \leq 2k$. } \\
& \text{ ii) any prime $p>w$ dividing both $[d_{\ell_1},e_{\ell_1}]$ and $[d_{\ell_1},e_{\ell_1}]$ must lie in $P_{\ell_1,\ell_2}$. }
\end{split}
\end{align}

Since $f$ is a compactly supported, smooth function, we may apply Fourier inversion to write
\begin{align}\label{inversion}
\exp \left( \sum_{j=1}^{2k} t_j \right) f(\underline{t})=\int_{\R^{2k}} \exp \left(-i \sum_{j=1}^{2k} t_j \xi_j \right) g(\underline{\xi}) \ d\underline{\xi} \qquad  \qquad  \underline{t}=(t_1,...,t_{2k}) \in  \R^{2k}.
\end{align}
for some smooth $g:\R^{2k} \rightarrow \R$ obeying decay estimates of the form $g(\underline{\xi}) \ll_A (1+\left\| \underline{\xi} \right\|)^{-A}$ for any $A>0$. This leads to the expression
$$f_R \left(\underline{d} \right)=f \left( \frac{\log d_1}{\log R},...,\frac{\log d_{2k}}{\log R} \right)=\int_{\R^{2k}} \frac{ g(\underline{\xi} )}{\prod_{j=1}^{2k} d_j^{\frac{1+i  \xi_j}{\log R} }} \ d\underline{\xi} .$$
Thanks to this integral representation, we find that the main term of (\ref{satz1}) can be recast as 
\begin{equation}\label{mainterm1fourier}
\frac{N}{W} \int_{\R^{2k}} \int_{\R^{2k}} \widetilde{K} \left(\underline{\xi}, \underline{\xi'} \right) g\left( \underline{\xi} \right) 
g\left(\underline{\xi'} \right) \ d \underline{\xi} \ d \underline{\xi'},
\end{equation}
where
$$\widetilde{K} \left(\underline{\xi}, \underline{\xi'} \right)=\sum_{ \underline{d},\underline{e}}^{\dagger} \frac{1}{lcm( [d_1,e_1],...[d_{2k},e_{2k}])}
\prod_{j=1}^{2k} \frac{\mu(d_j)\mu(e_j)}{ d_j^{\frac{1+i  \xi_j }{\log R}} e_j^{\frac{1+i  \xi_j' }{\log R} }}.$$

The sum $\widetilde{K}$ may be expressed as an Euler product
\begin{equation}\label{localK}
\widetilde{K} \left(\underline{\xi}, \underline{\xi'} \right)=\prod_{p\nmid W } K_p \left(\underline{\xi}, \underline{\xi'} \right), \qquad 
K_p=1+\frac{T_p}{p},
\end{equation}
where
\begin{equation}\label{Tpdef}
T_p= \sum_{ \substack{  \underline{d},\underline{e}    \\  [d_1,...,d_{2k},e_1,...,e_{2k}]=p }}^{\dagger} 
\prod_{j=1}^{2k} \frac{\mu(d_j)\mu(e_j)}{d_j^{\frac{1+i  \xi_j }{\log R}} e_j^{\frac{1+i  \xi_j' }{\log R} }}.
\end{equation}

\begin{lemma}\label{Kzeta} With notation as above we have the asymptotic formula
\begin{equation}\label{Kproducteval}
\widetilde{K} \left(\underline{\xi}, \underline{\xi'} \right)=(1+o(1)) \cdot \mathfrak{S} \cdot
\prod_{j=1}^{2k} \frac{\zeta_W(1+\frac{2+ \xi_j +i\xi'_j}{\log R} )}{  \zeta_W(1+\frac{1+i\xi_j}{\log R})  \zeta_W(1+\frac{1+i\xi_j}{\log R})    }
\end{equation}
provided that $|\xi_i|,|\xi_j'|\leq \sqrt{\log N}$ for each $i,j \leq 2k$. Here we have used the restricted zeta function $\zeta_W(s)=\prod_{p \nmid W}(1-p^{-s})^{-1}$. The singular series $\mathfrak{S}$ was defined in \eqref{singular}.
\begin{proof}
Let us first record the straightforward estimate $\prod_{p>w   } (1+O_k(p^{-2}) )=1+o(1)$, for later use. The local factors $T_p$ may be treated in one of two ways, depending on whether $p \in P_{\ell_1,\ell_2}$ or $p \in P_0$.\\
{\bf Case I: $p \in P_0$.} We see that $p$ must divide exactly one of the $[d_j,e_j]$ on the RHS of \eqref{Tpdef} and hence
\begin{equation}\label{Tp0}
T_p= T_p^{(0)}=
\sum_{j=1}^{2k} \left( -p^{-\frac{1+i  \xi_j }{\log R} } - p^{-\frac{1+i  \xi_j' }{\log R} }+p^{-\frac{2+i  \xi_j + i \xi_j' }{\log R} } \right).
\end{equation} 
In this case we find that
\begin{equation}\label{Kp0}
K_p= (1+O_k(p^{-2}) )\prod_{j \leq 2k} 
\frac{ \Big(1-p^{-1- \frac{1+i  \xi_j }{\log R} } \Big)  \Big(1-p^{-1- \frac{1+i  \xi_j' }{\log R} } \Big)  }{ \Big(1-p^{-1- \frac{2+i  \xi_j +i \xi_j'}{\log R} } \Big) } 
\end{equation}

{\bf Case II: $ p \in P_{\ell_1,\ell_2}$.} Here we must account for two types of configurations. Either $p$ divides exactly one of the $[d_j,e_j]$ or else $[d_{\ell_1},e_{\ell_1}]=p=[d_{\ell_1},e_{\ell_1}]$ and $p$ divides no other $[d_i,e_i]$. Accordingly, we have that
$$ T_p=T_p^{(0)}+  T_p^{(\ell_1,\ell_2)}$$
with $T_p^{(0)}$ as in \eqref{Tp0} and 
\begin{align}\label{Tpl1l2}
T_p^{(\ell_1,\ell_2)}&= p^{\frac{ 4+i (\xi_{\ell_1}+  \xi_{\ell_1}'+ \xi_{\ell_2}+ \xi_{\ell_2}' ) }{\log R}} 
+  p^{-\frac{2+i  \xi_{\ell_1} + i \xi_{\ell_2} }{\log R} }+p^{-\frac{2+i  \xi_{\ell_1} + i \xi_{\ell_2}' }{\log R} }
+ p^{-\frac{2+i  \xi_{\ell_1}' + i \xi_{\ell_2} }{\log R} }+p^{-\frac{2+i  \xi_{\ell_1}' + i \xi_{\ell_2}' }{\log R} } \notag\\
- & p^{-\frac{3+i  \xi_{\ell_1} + i \xi_{\ell_1}'+ i \xi_{\ell_2} }{\log R} }- p^{-\frac{3+i  \xi_{\ell_1} + i \xi_{\ell_1}'+ i \xi_{\ell_2}' }{\log R} }
-  p^{-\frac{3+i  \xi_{\ell_1} + i \xi_{\ell_2}+ i \xi_{\ell_2}' }{\log R} }- p^{-\frac{3+i  \xi_{\ell_1}' + i \xi_{\ell_2}+ i \xi_{\ell_2}' }{\log R} }.
\end{align}  
Since 
$$ 1-p^{-1-\frac{1+i a }{\log R} }=1-1/p+O\left( \frac{\log p}{p \sqrt{\log R}} \right)=(1-1/p) \left( 1+O\left(\frac{\log p}{p \sqrt{\log R}} \right) \right)$$
uniformly over all $|a|\leq \sqrt{\log R}$ and all $p \leq N^2$, we conclude that
\begin{equation}\label{Kpnot0}
K_p= \frac{ (1+O_k(p^{-2}) )}{(1-1/p)}  \Big( 1+O\Big(\frac{\log p}{p \sqrt{\log R}} \Big) \Big) \prod_{j \leq 2k} 
\frac{ \Big(1-p^{-1- \frac{1+i  \xi_j }{\log R} } \Big) 
\Big(1-p^{-1- \frac{1+i  \xi_j' }{\log R} } \Big)  }{ \Big(1-p^{-1- \frac{2+i  \xi_j +i \xi_j'}{\log R} } \Big) } 
\end{equation}
in case II.
Finally, assuming that $k$ grows sufficiently slowly with $N$, we may apply Mertens' Theorem to find that
\begin{align*}
\prod_{\substack{(\ell_1,\ell_2) \\ 1 \leq \ell_1 < \ell_2 \leq 2k}} \  
\prod_{p \in P_{\ell_1, \ell_2 }}  \Big( 1+O\Big(\frac{\log p}{p \sqrt{\log R}} \Big) \Big) 
&= \exp \Big(O \Big( \sum_{\substack{(\ell_1,\ell_2) \\ 1 \leq \ell_1 < \ell_2 \leq 2k}} \  
\sum_{p \in P_{\ell_1, \ell_2 }} \frac{\log p}{p \sqrt{\log R}} \Big) \Big)  \\
&=  \exp \Big( O \Big( \sum_{p \leq 4k^2 \log N}  \frac{\log p}{p \sqrt{\log R}} \Big) \Big) =1+o(1),
\end{align*}
since each $P_{\ell_1, \ell_2 }$ contains at most $O(\log N)$ primes. Inserting the product representations \eqref{Kp0} and \eqref{Kpnot0} into \eqref{localK}, we get \eqref{Kproducteval}.
\end{proof}
\end{lemma}

Using \cite[Equation 41]{Poly} to evaluate the product of zeta factors on the RHS of \eqref{Kproducteval}, we get that 
\begin{equation}\label{Keval}
\widetilde{K}(\xi,\xi')=(1+o(1)) B^{-2k}  \prod_{j=1}^{2k} \frac{(1+i\xi_j)(1+i\xi'_j)}{2+ \xi_j +i\xi'_j}, \ \ \qquad B=\frac{\varphi(W)}{W} \log R
\end{equation}
provided that $|\xi_i|,|\xi_j'|\leq \sqrt{\log N}$ for each $i,j \leq 2k$. Observe that the complementary range of $(\underline{\xi},\underline{\xi'})$ makes a negligible contribution to the integral in \eqref{mainterm1fourier}, due to the rapid decay of $g$.\\    
To prove the identity
\begin{equation}\label{intdiffrep}
\int_{\R^{2k}} \int_{\R^{2k}}  \prod_{j=1}^{2k} \frac{(1+i\xi_j)(1+i\xi'_j)}{2+ i\xi_j +i\xi'_j} g(\xi) g(\xi') \ d\xi d\xi'=\int_{\R_{+}^{2k}} Df(t)^2 \ dt=I(f),
\end{equation}
divide the RHS of (\ref{inversion}) by $\exp(\sum_{i=1}^{2k} t_i)$ and differentiate the integrand with respect to each variable $t_i$. This gives 
$$Df(t)= \int_{\R^{2k}} \prod_{j=1}^{2k} (1+i\xi_j) \exp \left(- \sum_{r=1}^{2k} t_r (1+i\xi_r) \right) g(\underline{\xi}) \ d \underline{\xi}$$
which is then squared and integrated to get the desired representation \eqref{intdiffrep}. Combining \eqref{mainterm1fourier}, \eqref{Keval} and \eqref{intdiffrep}, we find the asymptotic 
\begin{equation}
\mathfrak{S} NW^{-1} B^{-2k} I(f)
\end{equation} 
for the main term in \eqref{satz1}. The error term in \eqref{satz1} is easily estimated by a $2k$-fold divisor bound, yielding  a contribution no greater than $O_{\varepsilon,k} ( \|f \|_{\infty}^{2} R^{2+\varepsilon})$ for any $\varepsilon>0$, which is clearly acceptable.

\subsection*{The prime sum with weight $w_f^2$}
We now turn to the evaluation of the prime sum appearing in \eqref{primenonprime}. It will be convenient to carry out the computations for the shift $h_1$. 
Let us recall the notation \eqref{disc}. Given any primitive residue class $a \bmod q$, write 
$$\sum_{ m \leq x, m \equiv a(q) } 1_{\mathcal{P}}(m) =\frac{li(x)}{\varphi(q)} + E(x,a,q).$$ 
Expanding the sum $\sum_{n \leq N}^{'} 1_{\mathcal{P} }(n+h_1) w(n)^2$, we find the expression  

\begin{align}\label{satz2}
 &\sum_{\underline{d}, \underline{e}}^{\star} \left(\prod_{j=1}^{2k} \mu(d_i)  \mu(e_i)\right) f_R \left(\underline{d} \right) f_R\left( \underline{e}\right) 
\sum_{ \substack{n \leq N\\ n \equiv -h_i ([d_i,e_i])  \\ n\equiv b(W)   }}  1_{\mathcal{P} }(n+h_1) \notag \\ 
=&\sum_{\underline{d}, \underline{e}}^{\star} \left(\prod_{j=2}^{2k} \mu(d_j)  \mu(e_j)\right) 
f \left(0, \frac{\log d_2}{\log R},...,\frac{\log d_{2k} }{\log R} \right) f \left(0, \frac{\log e_2}{\log R},...,\frac{\log e_{2k} }{\log R} \right)  \notag\\ 
& \times \left[ \frac{li(N+h_1) - li(h_1)}{\varphi(q(\underline{d}, \underline{e}))} + E\left(N+h_1, a(\underline{d}, \underline{e}), q(\underline{d}, \underline{e}) \right) +O(h_1)\right].
\end{align}
Note that, in the second and third line, the vectors $\underline{d}, \underline{e}$ are made up of $2k-1$ variables. Here $a(\underline{d}, \underline{e}) $ denotes the unique residue class $\bmod \ q(\underline{d}, \underline{e})$ satisfying all the congruence conditions imposed on $n$ in the first line. In addition to the congruence conditions i) and ii) at the beginning of this appendix, we must also impose the following restricton on the tuples $\underline{d}, \underline{e}$:
\begin{equation}\label{iii}
\text{ iii) Any prime $p>w$ which divides $[d_{\ell},e_{\ell}]$ may not lie in $P_{1,\ell}$. }
\end{equation}
The main term in (\ref{satz2}) is treated in almost exactly the same way as \eqref{satz1} and hence we will only give a sketch of the argument. Replacing $f$ with the function $f^{*}(t_2,...,t_{2k})=f(0,t_2,...,t_{2k})$ we define a corresponding  function $g^{*}$ as before and find the expression 
\begin{equation}\label{mainterm2fourier}
\frac{li(N+h_1) - li(h_1)}{\varphi(W)} \int_{\R^{2k-1}} \int_{\R^{2k-1}} \widetilde{L} \left(\underline{\xi}, \underline{\xi'} \right) g^{*}\left( \underline{\xi} \right) 
g^{*}\left(\underline{\xi'} \right) \ d \underline{\xi} \ d \underline{\xi'}.
\end{equation}

Here
\begin{equation*}
\widetilde{L} \left(\underline{\xi}, \underline{\xi'} \right)=\prod_{p\nmid W } L_p \left(\underline{\xi}, \underline{\xi'} \right), \qquad 
L_p=1+\frac{S_p}{p-1},
\end{equation*}
where
\begin{equation*}
S_p= \sum_{ \substack{  \underline{d},\underline{e}    \\  [d_2,...,d_{2k},e_2,...,e_{2k}]=p }}^{\dagger} 
\prod_{j=2}^{2k} \frac{\mu(d_j)\mu(e_j)}{d_j^{\frac{1+i  \xi_j }{\log R}} e_j^{\frac{1+i  \xi_j' }{\log R} }}.
\end{equation*}
The sum on the RHS is restricted to vectors $\underline{d},\underline{e}$ satisfying the conditions i)-iii). The local factors $S_p$ are analyzed in much the same way as the $T_p$ in Lemma \ref{Kzeta}. However, for $p \in P_{\ell_1, \ell_2 }$ it is important to distinguish between pairs $(1,\ell_2)$ and $(\ell_1,\ell_2)$ where $\ell_1,\ell_2 \neq 1$. Write
$$S_p^{(0)}=
\sum_{j=2}^{2k} \left( -p^{-\frac{1+i  \xi_j }{\log R} } - p^{-\frac{1+i  \xi_j' }{\log R} }+p^{-\frac{2+i  \xi_j + i \xi_j' }{\log R} } \right)$$    
and for any $\ell \neq 1$, set
$$S_p^{(0)}(\ell)=
\sum_{j=2, j \neq \ell}^{2k} \left( -p^{-\frac{1+i  \xi_j }{\log R} } - p^{-\frac{1+i  \xi_j' }{\log R} }+p^{-\frac{2+i  \xi_j + i \xi_j' }{\log R} } \right).$$
Taking in to account all three conditions i)-iii) one finds that 

\[\arraycolsep=1.4pt\def\arraystretch{1.5}
S_p =
\left\{
	\begin{array}{lll}
		S_p^{(0)} & \mbox{  if }  p \in P_0 \\
		S_p^{(0)}(\ell) & \mbox{  if }   p \in P_{1,\ell} \\
		S_p^{(0)}+ T_p^{(\ell_1,\ell_2)} & \mbox{  if }   p \in P_{\ell_1,\ell_2} \text{ with } \ell_1,\ell_2 \neq 1 
	\end{array}
\right.
\]
with $T_p^{(\ell_1,\ell_2)}$ as in \eqref{Tpl1l2}. It follows that 
\begin{equation*}
\widetilde{L} \left(\underline{\xi}, \underline{\xi'} \right)=(1+o(1)) \cdot \mathfrak{S} \cdot
\prod_{j=2}^{2k} \frac{\zeta_W(1+\frac{2+ \xi_j +i\xi'_j}{\log R} )}{  \zeta_W(1+\frac{1+i\xi_j}{\log R})  \zeta_W(1+\frac{1+i\xi_j}{\log R})    }
\end{equation*}
as long as $|\xi_i|,|\xi_j'|\leq \sqrt{\log N}$ for each $2 \leq i,j \leq 2k$. From this point onwards, the argument from the previous section carries through without change and we obtain a main term asymptotic to
$$\mathfrak{S} \ \frac{li(N+h_1) - li(h_1)}{\varphi(W)} B^{1-2k} J^{(1)}(f) \sim \delta \mathfrak{S}\frac{N}{W} B^{-2k} J_{2k}^{(1)}(f)= \delta \beta(N)J^{(1)}(f), $$
with $J$ as in \eqref{integrals}. For the remainder term $E$, we recall that the assumption EH$[1/2+2\epsilon_0]$ gives the bound \eqref{BomVinplus}:
\begin{equation*}
\sum_{q \leq N^{1/2 + 2\epsilon_0} }  \max_{a \in (\Z/q\Z)^{\times})} |E \left(N, a, q \right)| \ll_A \frac{N}{(\log N)^A} 
\end{equation*}
for any $A>0$ and all large $N$. Using this estimate in combination with the trivial uniform bound $\max_a |E(N,a,q)| \leq N/q +1$, we find that for any $A>0$, the error term in \eqref{satz2} is of size at most
\begin{align*}
\sum_{  \underline{d}, \underline{e} }  | \lambda_{\underline{d}}|  
|\lambda_{\underline{e}}|& \max_{a} |E(2N, a,  q(\underline{d}, \underline{e}) )| \\ 
&\ll   \| f\|^2_{\infty} \Big(\sum_{r \leq N^{1/2 + 5\epsilon_0/ 2} } \tau_{4k}(r)^2 (N/r+1)  \Big)^{1/2} 
\Big(\sum_{r \leq N^{1/2 + 5\epsilon_0/2}} \max_a E(N+h_1,a,r) \Big)^{1/2} \\
& \ll_{A,k} N(\log N)^{-A}.
\end{align*}
In the last line we used a standard divisor sum bound $\sum_{r \leq x} \tau_{s}(r)^2/r \ll_s (\log x)^{s^2} $. This concludes the proof of Lemma \ref{mainasymp}.
\\ 
{\bf Acknowledgements.} For the most part, the results in this paper were obtained in my doctoral thesis. I would like to thank Terence Tao for suggesting this problem and for many useful discussions throughout the project. I also extend my gratitude to Shagnik Das and Sam Miner for sharing their combinatorial insights.


\begin{thebibliography}{99}
\bibitem{Banks} W. D. Banks, T. Freiberg, J. Maynard,  \textsl{On limit points of the sequence of normalized prime gaps}, Proceedings of the London Mathematical Society 113.4 (2016): 515--539.
\bibitem{BakFrei} R. Baker, T. Freiberg. \textsl{Limit points and long gaps between primes.} The Quarterly Journal of Mathematics 67.2 (2016): 233--260.
\bibitem{Poly} D. H. J Polymath, \textsl{Variants of the Selberg sieve, and bounded intervals containing many primes}, Research in the Mathematical sciences 1.1 (2014): 12.
\bibitem{FI} J.B. Friedlander, H. Iwaniec. Opera de cribro. Vol. 57. American Mathematical Soc., 2010.
\bibitem{Gall} P.X. Gallagher, \textsl{On the distribution of primes in short intervals. Mathematika}, 23  (1976): 4-9.
\bibitem{GL} D. A. Goldston,  A. H. Ledoan, \textsl{Limit points of the sequence of normalized differences between
consecutive prime numbers}, Analytic number theory (Springer, Cham, 2015) 115–125.
\bibitem{HM} A. Hildebrand, H. Maier, \textsl{Gaps between prime numbers}, Proc. Amer. Math. Soc. 104 (1988) 1–9.
\bibitem{kst} T.~K\H{o}v\'ari, V.T.~S\'os and P.~Tur\'an,\textsl{On a problem of K. Zarankiewicz}, Colloquium Math., \textbf{ 3} (1954), 50--57.
\bibitem{Luc} T. \L uczak, \textsl{A note on the density of sum-free sets.} J. Comb. Theory, Ser. A 70.2 (1995): 334-336.
\bibitem{May} J. Maynard, \textsl{Small gaps between primes}, Ann. of Math.(2),  \textbf{181}(1) (2015) 383--413.
\bibitem{Mer} J. Merikoski, \textsl{Limit points of normalized prime gaps.}, Journal of the London Mathematical Society (2020).
\bibitem{Pintz0} J. Pintz, \textsl{The Bounded Gap Conjecture and bounds between consecutive Goldbach numbers.} Acta Arithmetica 155.4 (2012): 397-405.
\bibitem{Pintz} J.Pintz, \textsl{ Polignac Numbers, Conjectures of Erd\H{o}s on Gaps Between Primes, Arithmetic Progressions in Primes, and the Bounded Gap Conjecture.} In: From Arithmetic to Zeta-Functions. Springer, Berlin; Heidelberg; New York, (2016) 367--384. 
\bibitem{Pintz2} J. Pintz, \textsl{A note on the distribution of normalized prime gaps}, Acta Arith. 184 (2018) 413-–418.
\bibitem{Sound} K. Soundararajan, \textsl{The distribution of prime numbers}. In: Granville A., Rudnick Z. (eds) Equidistribution in Number Theory, An Introduction. NATO Science Series, vol 237. Springer, Dordrecht (2007).
\bibitem{Sound2} K. Soundararajan, \textsl{Small gaps between prime numbers: The work of Goldston-Pintz-Yildirim.} Bulletin of the American Mathematical Society 44.1 (2007): 1-18.
\bibitem{Sze} E, Szemer\'edi, \textsl{On sets of integers containing no $k$ elements in arithmetic progression}, Acta Arithmetica \textbf{27}, 199--245.
\end{thebibliography}
\end{document}